\documentclass[]{article}
   
\usepackage{graphicx}
\usepackage{xcolor}

\usepackage{amsmath}
\usepackage{amsfonts}
\usepackage{amssymb}
\usepackage{amsthm}

\usepackage{hyperref}

\providecommand{\keywords}[1]{{\it Keywords: } #1}
\providecommand{\ams}[1]{{\it AMS Classification: } #1}

\newtheorem{lemma}{Lemma}
\newtheorem{theorem}{Theorem}

\newcommand{\R}{\mathbb{R}}
\newcommand{\defv}[1]{#1 = (#1_1,\ldots,#1_{n})}
\renewcommand{\bar}[1]{\overline{#1}}
\def\diam{\mathrm{diam}}
\def\conct[#1,#2]{\mbox {${#1} \leftrightarrow {#2}$}}

\begin{document}
  
\title{On the lower bound of the sum of\\ the algebraic
  connectivity of a graph and its complement}

\author{Mostafa
  Einollahzadeh and
  Mohammad Mahdi Karkhaneei }

\date{\large \today }

\maketitle

  \begin{abstract}
    For a graph $G$, let $\mu_2(G)$ denote its second
    smallest Laplacian eigenvalue. It was conjectured that
    $\mu_2(G) + \mu_2(\overline{G}) \geq 1$, where $\bar{G}$
    is the complement of $G$. This conjecture has been
    proved for various families of graphs. Here, we prove
    this conjecture in the general case. Also, we will show
    that
    $\max\{\mu_2(G), \mu_2(\overline{G})\} \geq 1 -
    O(n^{-\frac 13})$, where $n$ is the number of vertices
    of $G$.
  \end{abstract}

  \ams{05C50}

  \keywords{Laplacian eigenvalues of graphs, Nordhaus-Gaddum
    type inequalities, Effective resistance, Laplacian
    spread}

  \section{Introduction}
  Let $G$ be a simple graph with $n$ vertices. The Laplacian
  of $G$ is defined to be $L := D - A$, where $A$ is the the
  adjacency matrix of $G$ and $D$ is the diagonal matrix of
  vertex degrees. The eigenvalues of $L$,
  \begin{equation*}
    0 = \mu_1(G) \leq \mu_2(G) \leq \cdots \leq \mu_{n}(G),
  \end{equation*}
  expose various properties of $G$. The ``algebraic
  connectivity'' of a graph $\mu(G) := \mu_2(G)$ is
  an efficient measure for connectivity of graphs.

  The Laplacian spread of a graph $G$ is defined to be
  $\mu_n(G) - \mu_{2}(G)$. It was conjectured
  \cite{MR2899742,MR2826144} that this quantity is at most
  $n-1$, or equivalently,
  $\mu_{2}(G) + \mu_2(\overline{G})$ is at least 1
  (since $\mu_n(G) = n - \mu_2(\bar{G})$).

  \paragraph{Conjecture.} For any graph $G$ of order
  $n\geq 2$, the following holds:
  \begin{equation*}
    \mu(G) + \mu(\bar{G}) \geq 1,
  \end{equation*}
  with equality if and only if $G$ or $\bar{G}$ is
  isomorphic to the join of an isolated vertex and a
  disconnected graph of order $n-1$.
  
  This conjecture was proved for trees \cite{MR2383736},
  unicyclic graphs \cite{MR2522991}, bicyclic graphs
  \cite{MR2605816}, tricyclic graphs
  \cite{MR2529789}, cactus graphs \cite{MR2723499},
  quasi-tree graphs \cite{MR2782229}, graphs with diameter
  not equal to 3 \cite{MR2826144}, bipartite graphs \cite{MR3262243}, and $K_3$-free
  graphs \cite{MR3506494}. Also, \cite{MR3834197} provided
  a constant lower bound for
  $\mu(G)+\mu(\bar{G})$, by proving that 
   $\max\{\mu(G),\mu(\bar{G})\} \geq \frac 25$.

  Here, we prove the conjecture, in the general case. The
  main idea of the proof is as follows. Suppose that $x$ and
  $y$, respectively, are eigenvectors of $L(G)$ and
  $L(\bar{G})$, corresponding to the eigenvalues
  $\mu_2(G)$ and $\mu_{2}(\bar{G})$, with
  $\|x\|_2=\|y\|_2=1$. We have { $\sum_i x_i=\sum_i y_i=0$ and }
  \begin{eqnarray*}
    \mu(G) + \mu(\bar{G}) &=&  {x^T L(G) x+y^T L(\bar{G}) y}\\
    &=& \sum_{\{i,j\} \in E(G)}(x_i-x_j)^2 +
        \sum_{\{i,j\} \notin E(G)}(y_i-y_j)^{2} \\
    &\geq& \sum_{i < j}\min\{(x_i-x_j)^2,(y_i-y_j)^2\}.
  \end{eqnarray*}
  So, it is enough to show that
  \begin{equation*}
    \sum_{i < j}\min\{(x_i-x_j)^2,(y_i-y_j)^2\} \geq 1.
  \end{equation*}
  Now, suppose that $M$ is the maximum of all of numbers
  $|x_i-x_j|^2$ and $|y_i-y_j|^2$, for every
  $1\leq i,j \leq n$. A simple algebraic identity (see
  Lemma~\ref{lem:identity}), shows that
  $ \sum_{i < j}\min\{|x_i-x_j|^2,|y_i-y_j|^2\} \geq \frac
  1M. $ So, if $M \leq 1$, the proof is complete. It remains
  to consider the case $M > 1$. We manage this case with
  some basic properties of the effective resistances between
  pairs of vertices of $G$ and a useful characterization of
  the algebraic connectivity of graphs due to Fiedler (see
  Lemma~\ref{lem:fiedler}).

  Moreover, here, we give an asymptotic lower bound for the
  maximum of $\mu(G)$ and $\mu(\bar{G})$ by proving
  the following theorem.
  \paragraph{Theorem.}For all simple graphs $G$ with $n$
  vertices,
  \begin{equation*}
    \max\{\mu(G), \mu(\bar{G}) \} \geq 1- O(n^{-\frac 13}).
  \end{equation*}
  Also, for each $n \geq 4$, there is a graph $G$ which has
  $n$ vertices {such that} the maximum of $\mu(G)$ and
  $\mu(\bar{G})$ is less than 1. So, if $c_n$ denote the
  minimum of $\max\{\mu(G), \mu(\bar{G})\}$ over all
  graphs $G$ with $n$ vertices, then $c_n=1-o(1)$.

  The organization of the remaining of the paper is as
  follows. In Section 2, we gives some preliminaries and
  notations, which is necessary. In Section 3, we present
  two main results of the paper.
  
  \section{Preliminaries and notations}
  Throughout this paper $G$ is a simple graph, with
  $n \geq 2$ vertices $V(G):=\{v_1,\ldots,v_n\}$ and edges
  $E(G)$. The notation $\{i,j\} \in E(G)$, for
  $i,j \in \{1,\ldots, n\}$, means that $v_i$ and $v_j$ are
  adjacent in $G$. Also, $A(G)$ denotes the adjacency matrix
  of $G$. $D(G)$ is the diagonal matrix of vertex degree. We
  denote the Laplacian matrix of $G$ by $L(G):= D(G) - A(G)$
  and its eigenvalues by
  $0=\mu_1(G) \leq \mu_2(G) \leq \cdots
  \leq\mu_n(G) \leq n$ and the algebraic connectivity of
  $G$ by $\mu(G) := \mu_2(G)$. $\bar{G}$ and
  $\diam(G)$, respectively, denote the complement of $G$ and
  the diameter of $G$. If $v$ is a vertex of $G$, then
  $N_G(v)$ denotes the set of vertices which are adjacent to
  $v$ in $G$. Recall that for each $1\leq i < n$, we have
  $\mu_{i+1}(G) = n-\mu_{n+1-i}(\bar{G})$.
  
  The ``effective resistance" between two vertices $v_r$ and
  $v_s$ in a graph $G$ is denoted by $R^G_{r,s}$ and defined
  by:
  \begin{equation*}
    \frac{1}{R^G_{r,s}}:=
    \min
    \sum_{\{i,j\} \in E(G)}(x_i-x_j)^2,
  \end{equation*}
  where the minimum runs over all $x \in \R^n$, with
  $x_r - x_s = 1$. 
{So for an arbitrary vector $x\in \R^n$ and $1\leq r<s\leq n$, we have
\begin{equation*}
\sum_{\{i,j\} \in E(G)}(x_i-x_j)^2\geq \frac{(x_r-x_s)^2}{R^G_{r,s}}.
\end{equation*}}
(In the nontrivial case $x_r-x_s\neq 0$, by dividing the vector $x$ by $x_r-x_s$, we can assume $x_r-x_s=1$ and use the defining formula for $R^G_{r,s}$.)

{In this paper we only use some very basic properties of the effective resistance which are listed in the following lemma:}\footnote{For more information about the effective resistance and its relation to the algebraic connectivity of graphs, see \cite{ESV}. For the equivalence of the above definition of effective resistance and the standard definition by the concepts of electrical circuits and Kirchhoff's circuit laws, see \cite{Lyashko}.}
\begin{lemma}{ \label{lem:res}
Let $G=(V,E)$ be a simple graph with $n\geq 2$ vertices. Then the following holds:}
\begin{itemize}
\item[(a)]
For every subgraph $H$ of $G$ which contains the vertices $v_1$ and $v_2$, $R^{G}_{1,2}\leq R^{H}_{1,2}$.

\item[(b)] ``Total resistance of parallel circuits": Let $G_1$ and $G_2$ be two subgraphs of $G$, such that
$V(G_1)\cap V(G_2)=\{v_1,v_2\}$ and $E(G)$ be the disjoint union of $E(G_1)$ and $E(G_2)$. Then
\[\frac{1}{R^{G}_{1,2}}=\frac{1}{R^{G_1}_{1,2}}+\frac{1}{R^{G_2}_{1,2}}.\]
\item[(c)] ``Total resistance of series circuits": For $n\geq 3$, let $G_1$ and $G_2$ be two induced subgraphs of $G$, such that $v_1\in V(G_1)$, $v_3\in V(G_2)$ and $V(G_1)\cap V(G_2)=\{v_2\}$. Suppose also that $E(G)$ be the disjoint union of $E(G_1)$ and $E(G_2)$. Then 
\[R^{G}_{1,3}=R^{G_1}_{1,2}+R^{G_2}_{2,3}\]
\end{itemize}

\end{lemma}
\begin{proof}{
The parts (a) and (b) are immediate from the defining equation of the effective resistance. For (c), we have:
\begin{eqnarray*}
\frac{1}{R^{G}_{1,3}} &=&  \min_{x\in \R^n, x_1=0, x_3=1}
    \sum_{\{i,j\} \in E(G)}(x_i-x_j)^2\\
    &=&  \min_{x\in \R^n, x_1=0, x_3=1}
    \sum_{\{i,j\} \in E(G_1)}(x_i-x_j)^2 +\sum_{\{i,j\} \in E(G_2)}(x_i-x_j)^2 \\
  &=& \min_{x_2 \in \R, x_1=0, x_3=1} \frac{(x_1-x_2)^2}{R^{G_1}_{1,2}}+ \frac{(x_2-x_3)^2}{R^{G_2}_{2,3}}\\
&=&  \min_{x_2 \in \R} \frac{x_2^2}{R^{G_1}_{1,2}}+ \frac{(x_2-1)^2}{R^{G_2}_{2,3}} \\
&=& \frac{1}{R^{G_1}_{1,2}+R^{G_2}_{2,3}}.
\end{eqnarray*}}

{(The equality of the second and the third line, uses the fact that $G_1$ and $G_2$ have only the vertex $v_2$ in common.)}
\end{proof}



  We conclude this section with a useful lemma, due to
  Fiedler.
  {
  \begin{lemma}[\cite{MR0387321}]\label{lem:fiedler}
  	Let $G =(V,E)$ be a connected graph. Then 
    $\mu(G)$ is positive and equal to the minimum of the function
     \begin{equation*}
     \varphi(x):= n \frac{\sum_{\{i,j\} \in E(G)}(x_{i}-x_{j})^2 }
     { \sum_{i<j}(x_{i}-x_{j})^{2}}
    \end{equation*}
     over all nonconstant $n$-tuples
    $\defv{x} \in \R^n$ (i.e. $n$-tuples which are not of the form $x_i = c, i = 1, \ldots, n$).
  \end{lemma}}

{Therefore  $\mu(G)$ is the largest real number for which the
    following inequality holds for every
    $\defv{x} \in \R^n$:
    \begin{equation}\label{mu}
      n \sum_{\{i,j\} \in E(G)}(x_{i}-x_{j})^2  \geq
      \mu(G) \sum_{i<j}(x_{i}-x_{j})^{2}.
    \end{equation}
}
  
  \section{Main results}
  \subsection{Sum of the algebraic connectivity of a graph
    and its complement}
  \begin{theorem}
    \label{thm:main}
    For any graph $G$ of order $n\geq 2$, the following
    holds:
    \begin{equation*}
      \mu(G) + \mu(\bar{G}) \geq 1,
    \end{equation*}
    with equality if and only if $G$ or $\bar{G}$ is
    isomorphic to the join of an isolated vertex and a
    disconnected graph of order $n-1$.
  \end{theorem}

  Note that for any connected graph $G$, $\mu(G)$ is a
  positive number. So, when $G$ is a disconnected graph,
  Theorem 1 can be reformulated as the following lemma:
  \begin{lemma}\label{lem:1}
    If $G$ is a disconnected graph, then
    $\mu(\bar{G}) \geq 1$, with equality if and only if
    $\bar{G}$ is the join of an isolated vertex and a
    disconnected graph.
  \end{lemma}

  \begin{proof}
    We know that {the biggest Laplacian eigenvalue of every graph is at most the number of its vertices} and  the  Laplacian eigenvalues of $G$ are
    the union of the eigenvalues of its connected components,
    where each of which has less than $n$ vertices. So, {$\mu_n(G) \leq n - 1$ and} 
    $\mu(\bar{G}) = n - \mu_n(G) \geq 1$, with
    equality if and only if $G$ is the disjoint union of an
    isolated vertex and a connected graph $H$ with $n-1$
    vertices, where $\mu_{n-1}(H) = n-1$. But, note that
    $\mu_{n-1}(H) = n-1 - \mu_2(\bar{H})$. So, we
    have $\mu_{n-1}(H) = n-1$, if and only if $\bar{H}$
    is disconnected. Therefore, we have equality
    $\mu(\bar{G}) = 1$, if and only if $\bar{G}$ is the
    join of an isolated vertex and a disconnected graph
    $\bar{H}$.
  \end{proof}
  
  \begin{lemma}
    \label{lem:identity} Let $x = (x_{1}, \ldots, x_{n})$
    and $y = (y_{1}, \ldots, y_{n})$ be two vectors in
    $\mathbb{R}^{n}$ with $\sum_{i}y_{i} = \sum_{i}x_{i}=0$.
    Then
    \begin{equation*}
      \sum_{i<j}(x_i-x_j)^2(y_i-y_j)^2\geq \|x\|^2\|y\|^2.
    \end{equation*}
  \end{lemma}
  \begin{proof}
    Denote $A=\sum_{i<j}(x_i-x_j)^2(y_i-y_j)^2$. We have
{    
\begin{eqnarray*}
      2A &=& \sum_{i\neq j}(x_i-x_j)^2(y_i-y_j)^2 \\
          &=& \sum_{i, j}(x_i-x_j)^2(y_i-y_j)^2  \quad \mbox{(The terms corresponding to values $i=j$ are zero)} \\
          &=& \sum_{i,j}(x_{i}^{2}+x_{j}^2-2x_ix_j)(y_{i}^2+y_{j}^{2}-2y_iy_j)     \\
         &=&  \sum_{i,j} (x_i^2y_i^2+x_j^2y_j^2) +  \sum_{i,j} (x_i^2y_j^2+x_j^2y_i^2)+ 4\sum_{i,j} x_ix_jy_iy_j \\
         && -2\sum_{i,j} \big( x_i^2y_iy_j+x_j^2y_iy_j+x_ix_jy_i^2+x_ix_jy_j^2 \big) \\
        &=&  2\sum_{i,j} x_i^2y_i^2+  2\sum_{i,j} x_i^2y_j^2+ 4\sum_{i,j} x_ix_jy_iy_j  -4\sum_{i,j} \Big( x_i^2y_iy_j+x_ix_jy_i^2 \Big) \\
&=& 2n\sum_i x_i^2y_i^2 +2\Big(\sum_ix_{i}^2\Big)\Big(\sum_j y_{j}^2\Big)+4\Big(\sum_i x_iy_i \Big)\Big(\sum_jx_jy_j\Big) \\
        && -4  \Big(\sum_i x_i^2y_i\Big)\Big(\sum_jy_j\Big) -4\Big(\sum_i x_iy_i^2\Big)\Big(\sum_jx_j\Big) \\
         &=&2n\sum_{i}x_i^2y_i^2+2\Big(\sum_ix_{i}^2\Big)\Big(\sum_iy_{i}^2\Big)+
             4\Big(\sum_{i}x_iy_i\Big)^2 \quad {(\mbox{by $\sum_j x_j=\sum_j y_j=0$})}\\
         & \geq& 2 \|x\|^2\|y\|^2.
    \end{eqnarray*}
}
  \end{proof}
  \begin{lemma}\label{lem:diameter}
    Suppose that $G$ is a connected graph with $n\geq2$
    vertices $\{v_1, \ldots, v_n\}$. Let
    $x=(x_1,\ldots,x_n)$ be an {eigenvector} of $L(G)$
    corresponding to $\mu_2(G)$. If
    \begin{equation*}
      x_1=\max_ix_i, \quad  x_2=\min_ix_i,
    \end{equation*}
    and the distance between $v_1$ and $v_2$ is at most
    equal to $2$, then $\mu(G) \geq 1$. In particular,
    for any graph with diameter less than $3$, we have
    $\mu(G) \geq 1$.
  \end{lemma}
  \begin{proof}
    Since $x$ is orthogonal to the constant-entries vector
    $(1,\ldots,1)$, $\sum_ix_i=0$. So, $x_2<0<x_1$. Now,
    from $L(G)x = \mu x$, we get
    $\mu x_1=\sum_{\{i,1\}\in E(G)}(x_1-x_i)$. But, for
    each $i$, $x_1-x_i\geq 0$. Thus, for every
    $v_i\sim v_1$, we have $x_1-x_i\leq \mu x_1$, or
    equivalently, $x_i\geq (1-\mu)x_1$. Similarly, for
    every $v_j\sim v_2$, we have $x_j \leq (1-\mu)x_2$.
    Therefore, if $\mu<1$, for every $v_i \sim v_1$ and
    $v_j \sim v_2$ we have $x_j< 0 < x_i$. So $v_1$ and
    $v_2$ are not adjacent and have no common neighbor.
    {Thus,} the distance between $v_1$ and $v_2$ is greater
    than $2$.
  \end{proof}
  
  \paragraph{Proof of Theorem 1.}
  According to Lemma \ref{lem:1}, we can suppose that both
  $G$ and $\bar{G}$ are connected. {So, we have $n>2$}. Denote $\mu_2(G)$ and
  $\mu_2(\bar{G})$, respectively by $\mu$ and
  $\bar{\mu}$. Let $\defv{x}$ be a {normalized} eigenvector of
  $L(G)$ corresponding to $\mu_{2}(G)$ and $\defv{y}$ be
  a {normalized} eigenvector of $L(\bar{G})$ corresponding to
  $\mu_2(\bar{G})$. Note that $y$ is also an eigenvector
  of $L(G)$ corresponding to $\mu_n(G)$. So $x$ and $y$
  are two orthonormal vectors in $\R^n$ which are orthogonal
  to $e = (1,\ldots,1)$ (the eigenvector of
  $\mu_1(G) = 0$). Now, without loss of generality, we
  can suppose that
  $\max_{i<j}|x_i-x_j| \geq \max_{i<j}|y_i -y_j|$ and
  \begin{equation*}
    x_1=\max_{1\leq i \leq n} x_i, \quad x_2= \min_{1\leq i \leq n} x_i.
  \end{equation*}
  Note that since $\sum_i x_i=0$ and $x$ is nonzero, we have
  $x_2<0<x_1$.
  \paragraph{Step 1 (the case $x_1 - x_2 < 1$).}
  We have
  \begin{eqnarray*}
    \mu + \bar{\mu} &=& \sum_{\{i,j\} \in E(G)}(x_i-x_j)^2+\sum_{\{i,j\}\in E(\bar{G})}(y_i-y_j)^2\\
                            &\geq& \sum_{i<j}\min\{(x_i-x_j)^2, (y_i-y_j)^2\} \\
                            &\geq& \frac{\sum_{i<j}(x_i-x_j)^2 (y_i-y_j)^2}{\max_{i<j} \max \{(x_i-x_j)^2, (y_i-y_j)^2\}} \\
                            &\geq& \frac{ \|x \|^2 \| y \|^2}{(x_{1}-x_2)^2}
                            \qquad \mbox{(by Lemma \ref{lem:identity})} \\
                            &=& \frac{1}{(x_{1}-x_2)^2}. \qquad \mbox{{(since  $x$ and $y$ are normalized)}}
  \end{eqnarray*}
  The first inequality is implied by the fact that, for every pair  $i<j$, $\{i,j\}$ is an edge for exactly one of $G$ or $\bar{G}$. For the second inequality, note that:
  \begin{eqnarray*}
   (x_i-x_j)^2(y_i-y_j)^2  \!\!\!\!\! &=& \!\!\!\min\{(x_i-x_j)^2, (y_i-y_j)^2\}\max\{(x_i-x_j)^2, (y_i-y_j)^2\} \\
   &\leq & \!\!\!\! \min\{(x_i-x_j)^2, (y_i-y_j)^2\}{\max_{k<l} \max \{(x_k-x_l)^2, (y_k-y_l)^2\}}.
  \end{eqnarray*}
  
    Thus, in the case $x_1 - x_2 < 1$, we have
  $\mu + \bar{\mu} > 1$ and the proof is complete.
  So, we can suppose that $x_1-x_2 \geq 1$.

  \paragraph{Step 2 (the case $x_1-x_2 \geq 1$).}
  Since $\mu$ and $\bar{\mu}$ are positive, if one
  of them is greater than or equal to $1$ then
  $\mu + \bar{\mu} > 1$. So, we can suppose that
  $\mu, \bar{\mu} < 1$. Let $d$ be the distance
  between $v_1$ and $v_2$ in $G$.

  \paragraph{Step 2.1 ($d \leq 2$ or $d > 3$).} If
  $d \leq 2$ then, by Lemma \ref{lem:diameter}, we have
  $\mu \geq 1$. On the other hand, if $d > 3$ then
  $\diam(G) > 3$. So the diameter of $\bar{G}$ is at most
  equal to $2$. Thus, by Lemma \ref{lem:diameter},
  $\bar{\mu}\geq 1$. Therefore, we can suppose that
  $d=3$.

  \paragraph{Step 2.2 ($d = 3$).} Now, suppose that
  $s \geq 1$ is the maximum number of {vertex-}disjoint paths with
  length 3 between two vertices $v_1,v_2$ in $G$.

  Let $S$ be the union of the vertices of $s$ disjoint paths
  between $v_1$ and $v_2$ with length $3$. Denote
  $S_1:=S\cap N_G(v_2)$ and $S_2:=S\cap N_G(v_1)$ (see Figure \ref{fig:G}). Note that
  because $v_1$ and $v_2$ have no common neighbor in $G$,
  the vertices of $S_1$ are not adjacent to $v_1$ and the
  vertices of $S_2$ are not adjacent to $v_2$. Also,
  $|S_1|=|S_2|=s$ and $S = S_1 \cup S_2 \cup\{v_1,v_2\}$.
  Therefore, if we denote $A := N_G(v_1)\setminus S$,
  $B := N_G(v_2)\setminus S,$ and
  $C:=V(G) \setminus (A\cup B \cup S)$ then
  $\{A, B, C, S_1, S_2, \{v_1,v_2\}\}$ is a partition of the
  vertices of $G$ and because {by the maximality of $s$,} there is no path of length {at most} 3
  between $v_1$ and $v_2$ in $G\setminus (S_1\cup S_2)$, no
  vertex of $A\cup\{v_1\}$ is adjacent to any of the
  vertices of $B\cup\{v_2\}$. 
  \begin{figure}[h!]
    \centering
    \includegraphics[width=.7\linewidth]{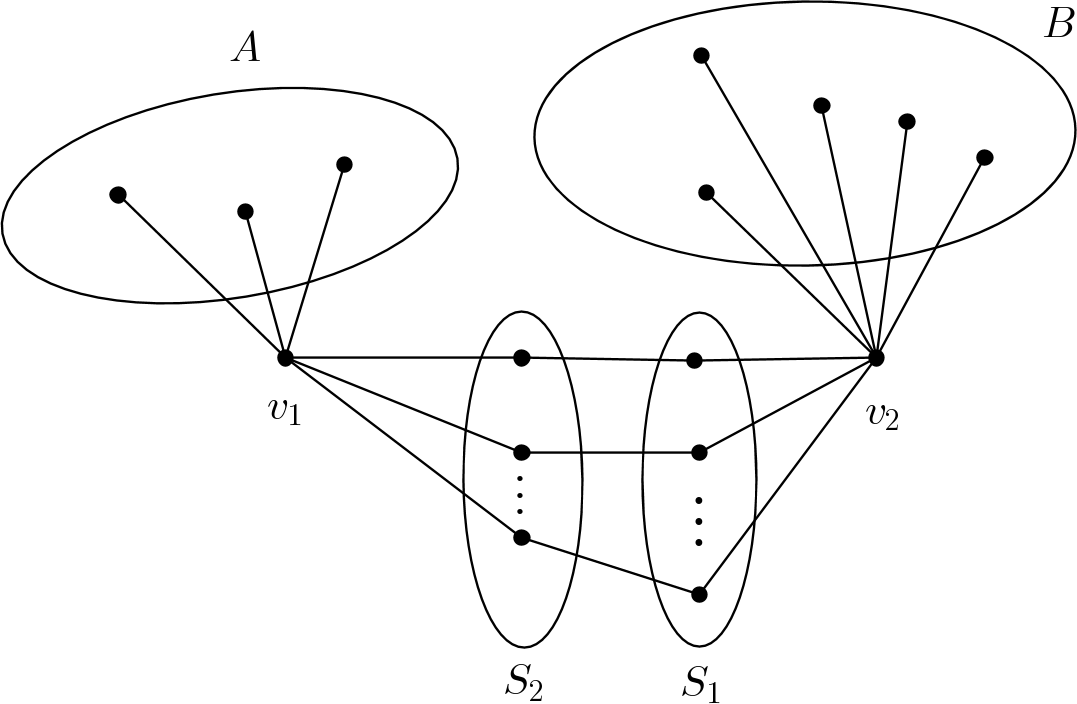}
    \caption{The subsets $S_1, S_2, A$ and $B$  of the vertices of $G$. }
    \label{fig:G}
  \end{figure}

Now, if we let
  $a := |A|, b := |B|$ and $c := |C|$, then
{\begin{equation}\label{n}
n = a + b + c + 2s+2.
\end{equation}}
 Furthermore we can suppose that
  $a \leq b$.

    Suppose that $l$ is the maximum of $|N_G(v) \cap A |$ for
  all $v\in S_1$ and $|N_G(u) \cap B |$ for all $u\in S_2$.
  So $G$ contains a subgraph $G_1$, as illustrated in Figure
  \ref{fig:L}. We have
 \[\mu \geq \frac{(x_1-x_2)^2}{R^G_{1,2}} \geq \frac{1}{R^G_{1,2}}.\]
 
 \begin{figure}[]
    \centering
    \includegraphics[width=.5\linewidth]{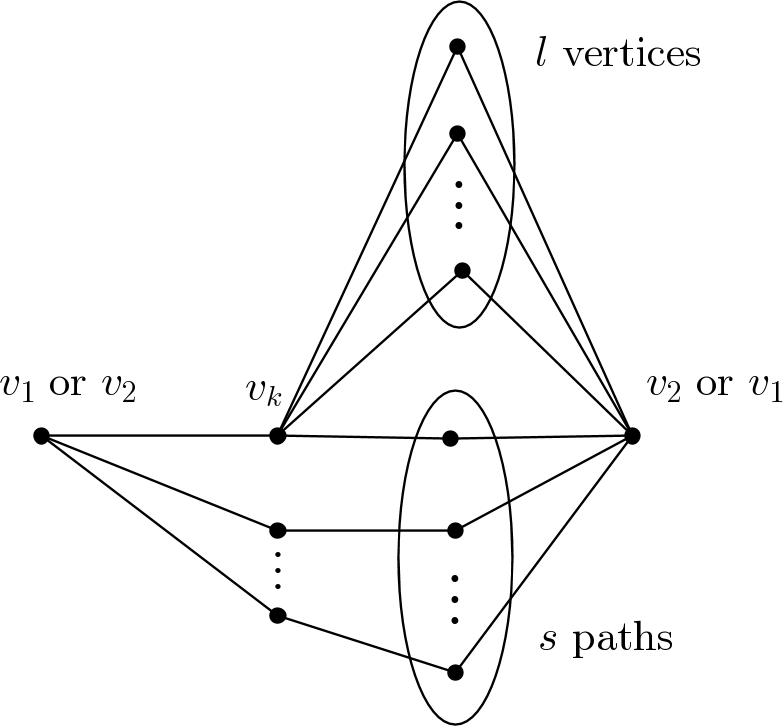}
    \caption{$G_1$: A subgraph of $G$.}
    \label{fig:L}
  \end{figure}
  
{ By part (a) of Lemma \ref{lem:res}, $R^G_{1,2}\leq R^{G_1}_{1,2}$. By multiple usage of the rules of total resistance for parallel and series circuits in Lemma \ref{lem:res}, 
we have:}
\[ \frac{1}{R^{G_1}_{1,2}}=\frac{1}{1+\frac{1}{\underbrace{\frac{1}{2} + \cdots + \frac{1}{2}}_{l+1 \text{ times}}}}  +\underbrace{\frac{1}{3} + \cdots + \frac{1}{3}}_{s-1 \text{ times}} +
    =\frac{1}{1+\frac{2}{l+1}}+\frac{s-1}{3} = \frac{l+1}{l+3}+\frac{s-1}{3} .\]
{(The terms $\frac{1}{2}$ correspond to $l+1$ parallel paths of length $2$ between $v_k$ and $v_2$ or $v_1$ in $G_1$.  Also the terms $\frac{1}{3}$ correspond to $s-1$ parallel paths
 of length $3$  not including $v_k$  between $v_1$ and $v_2$ in $G_1$, as in the Figure  \ref{fig:L}).}

{Finally we have a lower bound for $\mu$ in terms of $s$ and
  $l$:}
 \begin{equation}
    \label{eq:9}
    \mu \geq \frac{s-1}{3} + \frac{l+1}{l+3}.
  \end{equation}

  Next, we give a lower bound for $\bar{\mu}$. Note that
  $\bar{G}$ contains the edges of the graph illustrated in
  Figure \ref{fig:Gbar}, {which hereafter is called $G_2$}. Also, according to the definition
  of $l$, {each vertex of $S_1$ are adjacent to at most $l$ vertices of $A$   in $G$ and 
each vertex of $S_2$ are adjacent to at most $l$ vertices of $B$  { in $G$}.} 

\begin{figure}[h!]
    \centering \includegraphics[width=.7\linewidth]{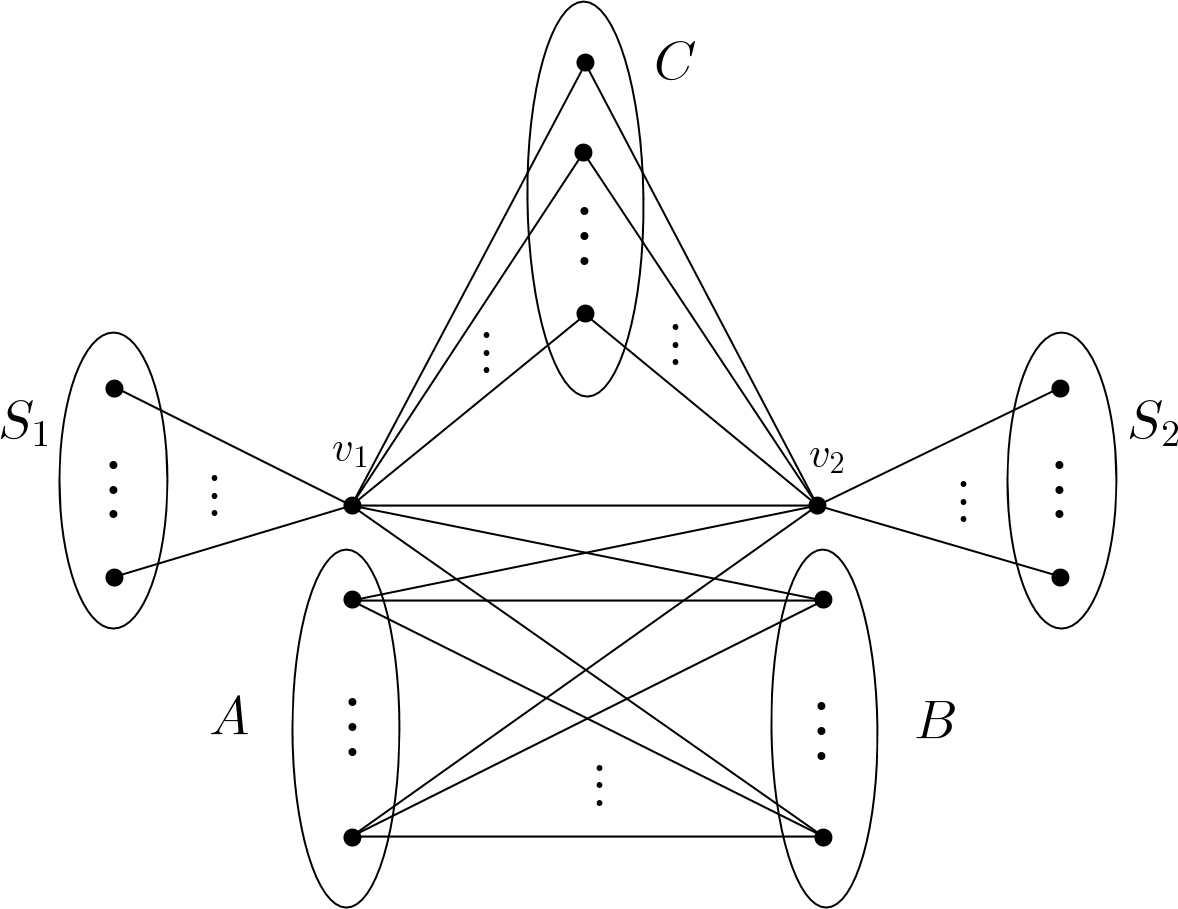}
    \caption{$G_2$: A subgraph of $\bar{G}$.}
    \label{fig:Gbar}
  \end{figure}    

Now,
  we define the subgraphs $H_1, H_2$ and $H_3$ of $\bar{G}$
  as follows: (see figure \ref{H})

\begin{figure}[h!]

    \centering \includegraphics[width=1.1\linewidth]{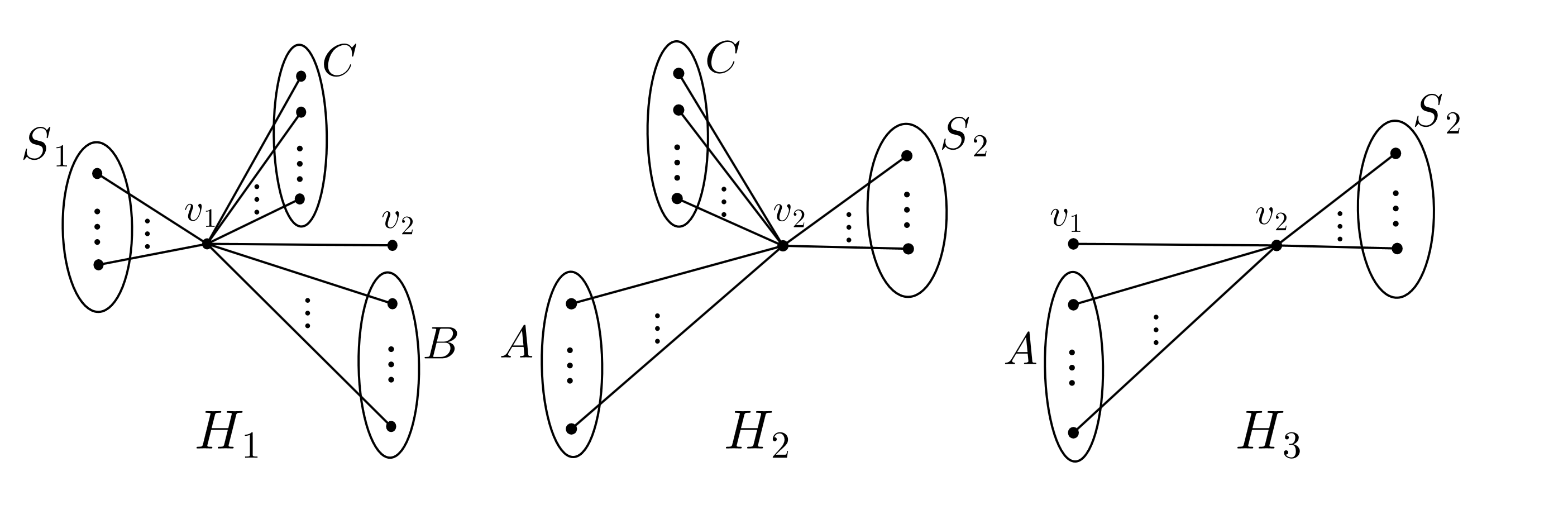}
    \caption{Graphs $H_1, H_2 $ and $H_3$.}
    \label{H}
  \end{figure}

  \begin{enumerate}
  \item $H_1$ is the union of the edges which join $v_1$ to
    the vertices of $\{v_2\} \cup B \cup C \cup S_1$,
  \item $H_2$ is the union of the edges which join $v_2$ to
    the vertices of $A \cup C \cup S_2$,
  \item $H_3$ is the union of the edges which join $v_2$ to
    the vertices of $\{v_1\} \cup A \cup S_2$.
  \end{enumerate}
  Next, note that each of $H_1, H_2,$ and $H_3$ is a star
  graph. Thus, $\mu(H_i) = 1,$ for $i=1,2,3.$ So, {if we define
\begin{eqnarray*}
X&=& \{(i,j):  (v_i,v_j) \in \mathrm{V}(H_1)^2\cup \mathrm{V}(H_2)^2\cup \mathrm{V}(H_3)^2\} \\
  &=& \{(i,j):  (v_i,v_j), (v_j,v_i) \not\in (S_1\times S_2)  \cup (S_1 \times A)   \cup (S_2  \times B) \cup (A\times B) \},
\end{eqnarray*}}
by  inequality \eqref{mu}, one can easily verify that
  \begin{equation}
    \label{eq:2}
    \begin{split}
       \sum_{\substack{(i,j) \in X\\
          i<j}}(y_i-y_j)^2 \quad \leq\quad  & (b+c+s+2) \sum_{\{i,j\} \in E(H_1)}(y_i-y_j)^2\  \\
      &+(a+c+s+1)\sum_{\{i,j\}\in E(H_2)}(y_i-y_j)^2\  \\
      &+(a+s+2)\sum_{\{i,j\}\in E(H_3)}(y_i-y_j)^2\ .
    \end{split}
  \end{equation}

  Since the resistance of a path with length $3$ is equal to
  $3$, we have
  \begin{align}
    \label{eq:3}
    \sum_{\substack{v_i \in S_1\\ v_j \in S_2}} (y_i-y_j)^2
    &\leq
      3 \sum_{\substack{v_i\in S_1\\ v_j \in S_2}}
    \big((y_i-y_1)^2 +(y_1-y_2)^2+(y_2-y_j)^2\big) \notag \\
    &=
      3s\sum_{v_i\in S_1}(y_1-y_i)^2+3s\sum_{v_j\in S_2}(y_2-y_j)^2+3s^2(y_1-y_2)^2.
  \end{align}
  Also, if we define $A_i:=A\setminus N_{\bar{G}}(v_i)$, for
  each $v_i \in S_1$, then $|A_i| \leq l$, for each
  $v_i\in S_i$, and
  \begin{equation}
    \label{eq:4}
    \begin{split}
      \sum_{v_i\in S_1}\sum_{v_j\in A_i}(y_i-y_j)^2 \ & \leq
      3\sum_{v_i\in S_1}\sum_{v_j\in A_i}
      \big( (y_i-y_1)^2 + (y_1-y_2)^2 + (y_2-y_j)^2 \big) \\
      & \leq 3sl \Big( \sum_{v_i\in S_1}(y_i-y_1)^2 +
      \sum_{v_j\in A}(y_2-y_j)^2 + (y_1-y_2)^2\Big).
    \end{split}
  \end{equation}
  Similarly, if we define
  $B_i := B \setminus N_{\bar{G}}(v_i)$, for each
  $v_i\in S_2$, we have
  \begin{equation}
    \label{eq:5}
    \sum_{\substack{v_i\in S_2\\v_j\in B_{i}}}(y_i-y_j)^2
    \leq 3sl \Big(
    \sum_{v_i\in S_2}(y_i-y_2)^2 +
    \sum_{v_j\in B}(y_1-y_j)^2 +
    (y_1-y_2)^2 \Big).
  \end{equation}
  On the other hand, for all pairs $(i,j)$ such that
  $v_i \in S_1$ and $v_j \in A \setminus A_i$, or
  $v_i \in S_2$ and $v_j \in B \setminus B_i$, or $v_i\in A$
  and $v_j \in B$, we have the trivial inequality
  \begin{equation}
    \label{eq:6}
    (y_i-y_j)^2 \leq (y_i-y_j)^2.
  \end{equation}

  Now, note that for all of the terms $(y_i-y_j)^2$ which
  appear in the {right} hand side of the
  inequalities~\eqref{eq:2},\eqref{eq:3},\eqref{eq:4},\eqref{eq:5},
  and \eqref{eq:6}, we have $\{i,j\}\in E(\bar{G})$, and on the
  other hand, for each $i<j$, one of $(y_i-y_j)^2$ or
  $(y_j-y_i)^2$ appears in the {left} hand side of one of
  these inequalities. Since $H_1$ and $H_2$ have no common
  edges, $b + c + s + 2 \geq a + c + s + 1$, and the terms
  $(y_i-y_j)^2$ which appear in \eqref{eq:6} does not appear
  in other inequalities, by summing up both sides of these
  inequalities, we will get
\begin{eqnarray}
    \label{eq:7}
        \sum_{i<j}(y_i-y_j)^2  
     & \leq&  \Big( (b+c+s+2) + (a+s+2) +3s^2  + 3sl + 3sl \Big) \nonumber \\
       && \times \sum_{\{i,j\}\in E(\bar{G})}(y_i-y_j)^2 \nonumber\\
       &=& (n + 2 + 3s^2 + 6sl)\sum_{\{i,j\}\in
        E(\bar{G})}(y_i-y_j)^2. \qquad{\mbox{(by \eqref{n})}}
 \end{eqnarray}
 
  { But $y=(y_1,\dots,y_n)$ is
  an eigenvector of $L(\bar{G})$ corresponding to
  $\bar{\mu}$ with $\sum_{i} y_i=0$,  so have
  $$\sum_{\{i,j\} \in E(\bar{G})}(y_i-y_j)^2 = y^T L(\bar{G}) \, y=\bar{\mu} \|y\|^2=  \frac{\bar{\mu}}{n} \sum_{i < j}(y_i-y_j)^2.$$}

 Therefore
  \begin{equation}
    \label{eq:8}
    \bar{\mu} \geq \frac{n}{n+2+3s^2+6sl}.
  \end{equation}
  \paragraph{A lower bound for $\mu + \bar{\mu}$.}
  Therefore, by inequalities \eqref{eq:9} and \eqref{eq:8},
  \begin{equation}
    \label{eq:10}
    \mu + \bar{\mu} \geq \frac{n}{n+2+3s^2+6sl}+
    \frac{s-1}{3} + \frac{l+1}{l+3}.
  \end{equation}

  Now by inequality \eqref{eq:9}, for $s \geq 3$, we have
  $\mu \geq 1$. So we can suppose that
  $s= 1\,\text{or}\, 2$.

  \paragraph{The case $ s = 1$.} In this case,
  \begin{equation*}
    \mu + \bar{\mu} \geq \frac{l+1}{l+3} + \frac{n}{n+5+6l}.
  \end{equation*}
  Thus, for each $n \geq 12$,
  \begin{equation*}
    \mu + \bar{\mu} \geq  1 - \frac{2}{l+3} + \frac{12}{17+6l}>1.
  \end{equation*}

  \paragraph{The case $s=2$.} Similarly, in this case,
  \begin{equation*}
    \mu + \bar{\mu} \geq \frac 13 + \frac{l+1}{l+3} + \frac{n}{n+14+12l}.
  \end{equation*}
  Note that, when $l \geq 3$,
  $\mu \geq \frac 13 + \frac{l+1}{l+3} \geq 1$. So,
  $\mu + \bar{\mu} > 1$. Also, for each of cases
  $l=0,1,2$, when $n \geq 10$,
  \begin{align*}
    l=0: & \qquad \mu + \bar{\mu} \geq \frac 23 + \frac{10}{24}  > 1, \\
    l=1: & \qquad \mu + \bar{\mu}  \geq \frac 13 +
           \frac 12 + \frac{10}{36}  > 1, \\
    l=2: & \qquad \mu + \bar{\mu} \geq \frac 13 + \frac 35 + \frac{10}{48} > 1.
  \end{align*}

{For all of the remaining cases,  the theorem can be checked numerically as follows: }

\begin{itemize}
\item {$n\leq 4$. The only connected graph $G$ with connected complement, is $P_4$ the path graph of length $3$. In this case:
$$\mu + \bar{\mu}=2\mu(P_4) = 2(2-\sqrt{2})>1.$$}

\item {$s=2, n\leq 9$. $G_2$ is a subgraph of $\bar{G}$ and so $\mu(G_2)\leq
\bar{\mu}$. The graph $G_2$ only depends on the nonnegative integers $a,b,c$ with constraints $s=2$ and $a+b+c+2s+2=n\leq 9$ (at most $n^3$ cases for each $n$). The minimum of the values of $\mu(G_2)$ in all of these cases is approximately $0.357$, and so by  \eqref{eq:9} we have
$$ \mu+\bar{\mu} > \frac{2}{3}+ 0.35 >1.$$}

\item {$s=1, \; 5\leq n\leq 11$. In these cases also the values of the second Laplacian eigenvalues of all graphs $G_2$ can be computed numerically and the minimum is approximately $0.429$. So again by \eqref{eq:9}, for $l\geq 2$ we have
$$\mu+\bar{\mu} > \frac{l+1}{l+3}+ 0.42 >1. $$}

{For the remaining cases $l=0,1$, $|S_1|=|S_2|=s=1$ and note that by the definition of $l$, the vertex in $S_1=\{u\}$ is adjacent in $G$ to at most $l$ vertices in $A$, and also the vertex in $S_2=\{v\}$ is adjacent in $G$  to at most $l$ vertices in $B$. }{Let $G_3$ be the graph obtained from $G_2$ by adding the edges in $\bar{G}$  between $u$ and $A$ and also edges between $v$ and $B$. }

{In the case $l=0$, $G_3$ is only dependent (up to isomorphism) on the size of the sets $A,B,C$ (at most $n^3$ cases for each $n$), and by computation for $5\leq n\leq 11$ the minimum of $\mu(G_3)$ is approximately $0.697$. Now $G_3$ is a subgraph of $\bar{G}$ and again by  \eqref{eq:9},
$$\mu+\bar{\mu} \geq \frac{1}{3}+\mu(G_3) > \frac13+0.69>1.$$}

{In the case $l=1$, $u$ is adjacent in $G_3$ to all vertices in $A$ except at most one vertex and also the same happens for $v$ and $B$. By deletion of edges (if necessary) we can suppose that $A=\emptyset$ or $u$ is nonadjacent in $G_3$  to exactly one vertex in $A$, and also the simillar thing for $v$ with respect to $B$. Again the graph $G_3$ depends only (up to isomorphism) on the size of the sets $A,B,C$  and  for $ n\leq 11$ the minimum of $\mu(G_3)$ is approximately $0.518$. So we have
$$\mu+\bar{\mu} \geq \frac{l+1}{l+3}+\mu(G_3) > \frac12+0.51>1.$$}
\end{itemize}

Thus, the theorem holds for every
  $n \geq 2$. \qed

  \subsection{Maximum of the algebraic connectivity of a
    graph and its compelement} As a by-product of the proof
  of Theorem 1, we prove the following theorem.
  \begin{theorem}
    For all graphs $G$ with $n$ vertices,
    \begin{equation*}
      \max\{\mu(G), \mu(\bar{G}) \} \geq 1- O(n^{-\frac 13}).
    \end{equation*}
  \end{theorem}
  \begin{proof}
    It is sufficient to prove the theorem for a connected
    graph $G$ with $n$ vertices, for sufficiently large
    integers $n$. With the notations and assumptions at the
    beginning of the proof of Theorem 1, we know
    \begin{equation*}
      \mu + \bar{\mu} \geq \frac {1}{(x_1-x_2)^2}.
    \end{equation*}
    So, if $(x_1-x_2)^2 \leq \frac 12$, then
    $\mu + \bar{\mu} \geq 2$ and the maximum of
    $\mu$ and $\bar{\mu}$ is at least $1$. Thus we
    can suppose that $(x_1-x_2)^2 \geq \frac 12$. Now,
    similar to Step 2 of the proof of Theorem 1, we can
    suppose that both $\mu$ and $\bar{\mu}$ are
    smaller than $1$. This implies that the distance between
    $v_1$ and $v_2$ in $G$ is equal to $3$. Let $s \geq 1$
    be the maximum number of vertex-disjoint paths with length 3
    between two vertices $v_1$ and $v_2$ in $G$. Therefore,
    \begin{equation*}
      1 > \mu \geq \frac{(x_1-x_2)^2}{R^G_{1,2}}
      \geq \frac 12 \times \frac s3 = \frac s6.
    \end{equation*}
    So $s \leq 5$. Now, by \eqref{eq:8}, with the notations
    which was defined in Step 2.2 of the proof of Theorem 1,
    we have
    \begin{equation*}
      \bar{\mu} \geq \frac{n}{n+2+3s^2+6sl}
      \geq \frac{n}{n + 80 + 30 l}.
    \end{equation*}
    On the other hand, note that according to the definition
    of $l$, $G$ contains {the subgraph $G_1$ which is} illustrated in
    Figure \ref{fig:L}. Without loss of generality, we can
    assume that the left vertex in Figure \ref{fig:L} is
    $v_1$ and the right vertex is $v_2$. Now we have
    \begin{equation*}
      \mu x_1 = \sum_{\{1,i\}\in E(G)}(x_1 - x_i) \geq (x_1-x_k).
    \end{equation*}
    So $x_k \geq (1-\mu)x_1$ and
    $x_k-x_2 \geq (1-\mu)(x_1-x_2)$. Therefore,
    \begin{equation*}
      1>\mu=\sum_{\{i,j\} \in E(G)}(x_i-x_j)^2 \geq \frac{(x_k-x_2)^2}{R_{k,2}^G} \geq (1-\mu)^2\times \frac 12 \times \frac{l+1}{2}. 
    \end{equation*}
    Thus, $(1-\mu)^2 \leq \frac{4}{l+1}$ and
    $\mu \geq 1- \frac {2}{\sqrt{l+1}}$. Now, we
    consider two cases:
    \begin{enumerate}
    \item $n < l ^{\frac 32}$. Thus
      $n^{\frac 13} < l^{\frac 12}$ and
      \begin{equation*}
        \mu \geq 1 - \frac{2}{\sqrt{l}} \geq 1 - \frac{2}{n^{\frac 13}}.
      \end{equation*}
    \item $n \geq l^{\frac 32}$. So $l \leq n^{\frac 23}$,
      $\frac ln \leq n^{-\frac 13}$, and
      \begin{equation*}
        \bar{\mu} \geq 1- \frac{80+30l}{n+80+30l} \geq 1- \frac{80}{n} - \frac{30l}{n} \geq 1 - \frac{110}{n^{\frac 13}}.
      \end{equation*}
    \end{enumerate}
    Therefore, in all cases we have
    $\max\{\mu,\bar{\mu}\} \geq 1-
    O(n^{-\frac{1}{3}})$.
  \end{proof}

  \paragraph{Remark.} For each $n \geq 4$, define a graph
  $G$ with vertices $\{v_1, \ldots, v_n\}$ such that the
  induced subgraph on $\{v_3,\ldots v_n \}$ is a complete
  graph, $v_1$ is only adjacent to $v_3$ and $v_2$ is only
  adjacent to $v_4$. One can observe that
  $$\mu(G) = \mu(\bar{G})= 
  \frac{n-\sqrt{n^2-4n+8}}{2}= 1 - \frac{1}{n} + O(\frac
  {1}{n^2})< 1.$$ Therefore, for each $n \geq 4$, there
  exist a graph $G$ which has $n$ vertices and the maximum
  of $\mu(G)$ and $\mu(\bar{G}) $ is less than
  1.


  \bibliographystyle{alpha}
  \bibliography{LaplacianSpread}

\vspace{10pt}
\noindent Mostafa  Einollahzadeh, Isfahan University of Technology, Isfahan, Iran

\noindent e-mail: m\_einollahzadeh@iut.ac.ir

\vspace{5pt}
\noindent
Mohammad Mahdi Karkhaneei, Hesaba Co., Tehran, Iran

\noindent e-mail: karkhaneei@hesaba.co

\end{document}